%% file: certified-galois.tex
\documentclass[11pt]{amsart}

\usepackage{amsmath}
\usepackage{hhline}
\usepackage{stackengine}
\usepackage{amsthm}
\usepackage{blkarray}
\usepackage{mathrsfs}
\usepackage[margin=1in]{geometry}
\usepackage{amssymb}
\usepackage{mathtools}
\usepackage{xcolor}
\usepackage{algorithm}
\usepackage{algpseudocode}
\usepackage{bbold}
\usepackage{cite}
\usepackage{latexsym}
\usepackage{booktabs}
\usepackage{caption}
\usepackage{siunitx} 
\usepackage{bm}
\usepackage{graphicx}

\usepackage{textcomp}
\usepackage{subfigure}
\usepackage[toc,page]{appendix}
\usepackage{url}
\usepackage{hyperref}
\usepackage{multicol}
\usepackage{multirow}
\usepackage{tikz}
\usetikzlibrary{3d}
\usepackage{tikz-3dplot}
\usepackage{pgfplots}
\usepackage{mathabx,epsfig}

\let\oldtextit\textit 
\renewcommand\emph[1]{\oldtextit{\color{blue}#1}}

\theoremstyle{definition}
\newtheorem{definition}{Definition}[section]

\usetikzlibrary{arrows}

\newcommand*\transitive{%
  \mathrel{%
    \begin{tikzpicture}[scale=0.2]%
     \draw[->>] (1,0) arc (180:0:1cm); 
    \end{tikzpicture}%
  }%
}

\usepackage{cleveref}
\newtheorem{proposition}[definition]{Proposition}

\newtheorem{lemma}[definition]{Lemma}

\newtheorem{thm}[definition]{Theorem}

\theoremstyle{remark}
\newtheorem{remark}[definition]{Remark}

\algrenewcommand\algorithmicrequire{\textbf{Input:}}
\algrenewcommand\algorithmicensure{\textbf{Output:}}
\AtBeginEnvironment{algorithmic}{%
}

\newcommand{\CC}{{\mathbb C}}

\newcommand{\mainfilecheck}[1]{1}
\DeclareMathOperator{\Mon}{Mon}
\DeclareMathOperator{\gw}{gw}

\makeatletter
\def\@settitle{\begin{center}%
  \baselineskip13\p@\relax
    \Large
\@title
  \end{center}%
}
\makeatother
\title{Certifying Galois/monodromy Actions via Homotopy Graphs}

\author{Timothy Duff}
\address{Department of Mathematics, University of Missouri - Columbia}
\email{tduff@missouri.edu}
\urladdr{https://timduff35.github.io/timduff35/}

\author{Kisun Lee}
\address{School of Mathematical and Statistical Science, Clemson University, 220 Parkway Drive, Clemson, SC 29634}
\email{kisunl@clemson.edu}
\urladdr{https://klee669.github.io}

\begin{document}

\begin{abstract}
We develop a certified numerical algorithm for computing Galois/monodromy groups of parametrized polynomial systems. Our approach employs certified homotopy path tracking to guarantee the correctness of the monodromy action produced by the algorithm, and builds on previous ``homotopy graph" frameworks. 
We conduct extensive experiments with an implementation of this algorithm, which we have used to certify properties of several notable Galois/monodromy groups which arise in several examples drawn from pure and applied mathematics.
\end{abstract}

\maketitle

\section{Introduction}

Many problems are naturally modeled as \emph{parametrized systems of algebraic equations}, \begin{equation}F(x;z)=0\label{eq:Fxz},\end{equation} with the property that for generic parameters $z$, the system $F$ has a fixed, finite number of complex-valued isolated solutions $x$.
Determining the structure of such problems and the difficulty of solving them are central themes in applied and computational algebraic geometry.

As is the case in other mathematical areas, the meanings of words like ``difficulty" and ``structure" are often clearest in hindsight. 
In our setup, the number of solutions is perhaps the simplest measure of a problem's difficulty.
However, a more refined measure is captured by the problem's \emph{monodromy group}, also known as the \emph{geometric Galois group}. 
The meaning of ``monodromy" in our setup is the traditional one: as the parameters $z$ move along a loop in parameter space, avoiding a branching locus, the solutions change continuously, eventually being permuted.

There is already a substantial body of work applying numerical methods to the computation of monodromy groups---see e.g.~~\cite{brazelton2024monodromy,duff2022galois,duff2023using,duff2025galois,pichon2025galois,hauenstein2018numerical,leysot,edwards2026computing,granboulan1996construction}.
Two characteristics shared by many of the methods developed in these works are (1) despite the potential use of numerical heuristics, the end result may be rigorously \textit{certified}, and (2) the monodromy group is calculated \textit{en route} to solving the problem~\eqref{eq:Fxz}, with little additional overhead.
\textit{In this work}, we present an approach to certified Galois/monodromy group computation that fuses these two characteristics.
Unlike previous works which strive to algorithmically determine the full Galois/monodromy group, we favor an experimental approach, which nevertheless rigorously certifies an action of some subgroup of the Galois group (\Cref{thm:correctness}). 
Even if only partial information about the group is recovered, this enables us to make rigorous claims about certain group invariants, e.g.~the \emph{Galois width} of~\cite{duff2025galois}.

Our methodology builds on previous works that compute Galois/monodromy groups using \emph{homotopy path tracking}. Despite being fast and flexible, floating-point tracking may cause \emph{path jumping}, i.e.\  switching to a different solution path, which produces incorrect correspondences and hence an incorrect group. Certifying a numerical approximation of a solution to $F$ \textit{a posteriori} (e.g., \cite{breiding2020certifying,burr2019effective,hauenstein2012algorithm,lee2019certifying,lee2024effective,krawczyk1969newton}) is insufficient in this context. We must ensure that the tracking process follows a single solution path. 
In this paper, we adapt certified path-tracking methods based on \emph{interval arithmetic}~\cite{doi:10.1137/0731048,martin2013certified,van2011reliable,xu2018approach,duff2024certified,guillemot2024validated,guillemot2025certified}, which construct a region that covers the solution path, guaranteeing that the tracking process follows a single solution path without jumping.

To compute monodromy permutations and recover the monodromy group, we revisit the graph-based framework of~\cite{duff2019solving}, where homotopy paths are assembled to form a \emph{homotopy (multi)graph} embedded in the parameter space. Each vertex is a generic point in the parameter space, and each edge between two vertices corresponds to a path connecting them. By tracking solutions along these edges in a certified manner, we rigorously obtain monodromy permutations.
Improvements over the framework first proposed in~\cite{duff2019solving} include: (1) certified path-tracking, enabling rigorous  Galois group computation  such as in \Cref{thm:correctness} and~\Cref{prop:galois-width}, (2) using \emph{spanning trees} to organize the generators of monodromy subgroups, (3) a \emph{saturation} strategy ensuring loops are fully explored.

The paper is organized as follows. In \Cref{sec:background}, we provide background on interval arithmetic, the Krawczyk method, and monodromy groups, and review fundamental algorithms for certified path tracking. \Cref{sec:homotopy-graph} describes the graph-based framework in complete detail: constructing parameter homotopies, certifying edge correspondences, and generating monodromy permutations from a (saturated) graph. Finally, \Cref{sec:experiments} presents experimental results using our \texttt{Julia} implementation, applying our approach to test cases ranging from computer vision to algebraic geometry, as well as to some univariate benchmarks.

\section{Background}\label{sec:background}

\subsection{Interval arithmetic and the Krawczyk method}
\emph{Interval arithmetic} extends standard arithmetic to intervals, executing conservative and reliable computations. For an operator $\odot$ and intervals $I,J$, we define $I \odot J := \{x\odot y \mid x\in I, y\in J\}$. For explicit formulas, and a discussion of issues related to floating-point (which we mostly ignore), we refer to the standard text \cite{moore2009introduction}.

We work over $\mathbb{C}$ by applying real interval arithmetic to real and imaginary parts independently, for instance $I:=\Re(I)+i\Im(I)$. For an interval vector $I=(I_1,\dots,I_n)\subset\mathbb{C}^n$, we define the \emph{norm}
\[
\|I\| := \max_{1\le i\le n} \max\{|\Re(I_i)|,\; |\Im(I_i)|\}
\]
where $|J| = \max\limits_{a\in J}|a|$ denotes the \emph{magnitude} of an interval $J$.
This models the real $\infty$-norm on $\mathbb{C}^n$.

Given an interval vector $I\subset\mathbb{C}^n$ and a function $f:\mathbb{C}^n\to\mathbb{C}$, an \emph{interval enclosure} of $f$ on $I$ is an interval $\square f(I)$ satisfying $\square f(I)\supset\{f(x)\mid x\in I\}$. Such enclosures are not unique in practice, as they depend on the strategy used to evaluate $f,$ and may be overly conservative.

Let $F:\mathbb{C}^n\to\mathbb{C}^n$ be a square system of polynomial or rational functions, and let $x\in\mathbb{C}^n$. For $r>0$ and an invertible matrix $A$, the \emph{Krawczyk operator} is
\[
K(F,x,r,A) := -AF(x) + (Id - A\square JF(x+rB)) rB
\]
where $B = ([-1,1]+i[-1,1])^n$.
We recall a quantitative variant of the Krawczyk test~\cite{krawczyk1969newton}.

\begin{thm}\label{thm:Krawczyk-short}\cite[Theorem 2.1]{guillemot2024validated} 
If $K(F,x,r,A)\subset r\rho B$ for some $\rho<1$, then the system $F(x)=0$ has a unique solution $x^*$ in the box $x+rB$, and $\|x-x^*\|\le r\rho$. Moreover, the quasi-Newton map $x\mapsto x-AF(x)$ is $\rho$-Lipschitz on $x+rB$.
\end{thm}

When the hypotheses of~\Cref{thm:Krawczyk-short} hold, we say $x$ is a \emph{$\rho$-approximate solution} with \emph{certification radius} $r$. In practice,~\Cref{thm:Krawczyk-short} may be applied by choosing $A \approx JF(x)^{-1}$, and the inclusion $K(F,x,r,A)\subset r\rho B$ may be tested by evaluating suitable interval enclosures $\square F$, $\square JF$. 

Closely related to~\Cref{thm:Krawczyk-short} is the \emph{Krawczyk test} of~\Cref{algo:Krawczyk-test}.
This test, mathematically equivalent to checking the inclusion in \Cref{thm:Krawczyk-short} after normalizing by the radius $r$, is used throughout our certified predictor-corrector framework.

\textit{Crucially}, the constants used to evaluate $F$ in the Krawczyk test may themselves be intervals. In this case, a test output of ``true" gives us a proof that there exists a unique root in $x+rB$ for any choice of values in these intervals.

\begin{algorithm}[ht]
\caption{KrawczykTest}\label{algo:Krawczyk-test}
\begin{algorithmic}[1]
\Require 
\begin{itemize}
    \item A square system $F=\{f_1,\dots, f_n\}$, possibly with interval constants 
    \item a point $x\in\mathbb{C}^n$, or more generally an interval $x \subset \mathbb{C}^n$,
    \item a radius $r>0$, \item an invertible matrix $A\in\mathbb{C}^{n\times n}$, and
    \item a parameter $\rho\in(0,1)$.
\end{itemize}
\Ensure A boolean indicating whether $x$ certifies a unique root in $x+rB$.
\State Compute $K := -\frac{1}{r} A F(x) + (I_n - A\square JF(x+rB))\, B$.
\State \Return $\|K\| < \rho$.
\end{algorithmic}
\end{algorithm}

\subsection{Galois/monodromy groups of polynomial systems}\label{sec:galois-monodromy}

Parametric polynomial systems often have the property that there is a finite number $d>0$ of complex solutions $x\in\mathbb{C}^n$ for generic choices of parameters $z\in\mathbb{C}^m$. 
The solutions are thus given locally as algebraic functions $x_1(z), \dots , x_d (z)$, defined implicitly via an incidence variety of ``problem-solution pairs",
\begin{equation}\label{eq:solution-variety}
X = \{(z,x) \in \mathbb{C}^m \times \mathbb{C}^n \mid \tilde{F}(x;z) = 0\}, 
\end{equation}
where $\tilde{F}(x;z)$ is a system of polynomial equations (not necessarily square) which define $X$ globally.
The variety $X$ is equipped with a projection $\pi : X \to \mathbb{C}^m$ which generically $d$-to-$1$.
Often, we may also assume that $X$ is irreducible, in which case we call $X$ the \emph{solution variety}.
The integer $d$ is usually known as the \emph{generic root count} or \emph{degree} of the problem of solving $\tilde{F}.$
For a generic problem instance $z$, the fiber $\pi^{-1}(z)$ consists of exactly $d$ isolated points.

As $z$ moves along loops in $\mathbb{C}^m$ avoiding the branch locus of $\pi$, the $d$ solutions undergo analytic continuation and are permuted. Assuming $X$ is irreducible, these permutations form a transitive subgroup $\Mon_\pi\hookrightarrow S_d$, called the \emph{monodromy group} of the problem.
This group is well-defined up to conjugacy in $S_d$. As is standard, we suppress its dependence on the generic point $z_0$.

\begin{remark}\label{remark:reducible}
For certain problems, the variety $X$ may be reducible, possibly even possessing multiple components $X_1 ,\dots , X_k$ such that the restrictions $\pi_k : X_k \to \CC^m$ are dominant.
We consider such an example in~\Cref{sec:27-lines}, where we discuss $27$ lines on a symmetric cubic surface.
In such cases, the monodromy group acts intransitively as a subgroup of $\Mon_{\pi_1}\times \cdots \times \Mon_{\pi_k}$.
\end{remark}

Assuming $X$ is irreducible, $\Mon_\pi$ coincides with the geometric \emph{Galois group} of the branched cover $\pi$. If $K$ is the normal closure of the rational function field $\mathbb{C}(X)$ over the function field $L=\mathbb{C}(\mathbb{C}^m)$, then $\Mon_\pi$ as an \textit{abstract group} is isomorphic to the Galois group of the finite extension of fields $K/L$~\cite[Section 1]{yahl}. However, here we are mainly concerned with the structure of $\Mon_\pi$ as a \textit{permutation group} acting on $\pi^{-1}(z).$

Understanding the monodromy group provides structural information about the solution set, such as transitivity, decomposability, or the presence of nontrivial symmetries. By numerically tracking solutions along loops, one directly obtains permutations of the fiber. \Cref{sec:homotopy-graph} develops a framework for computing these permutations using parameter homotopies and interval arithmetic.

\subsection{Certified path tracking}\label{subsec:certified-track} 
\emph{Homotopy path tracking}, also known as numerical continuation, is a standard numerical method for solving systems of nonlinear equations. 
Such methods construct a start system $G:\mathbb{C}^n \to \mathbb{C}^n$ with known solutions and a homotopy function $H(x;t): \mathbb{C}^n \times [0,1] \to \mathbb{C}^n$ such that $H(x;0) = G(x)$,  ideally with the property that any isolated, non-singular solution of $G$ extends to an analytic solution path $x(t) : [0,1] \to \CC^n$ with $H(x(t) ; t) = 0$ for all $t\in [0,1].$ 

In practice, tracking a solution path from $t = 0$ to $t = 1$ is carried out using a numerical predictor-corrector method~\cite[Chapter 2.3]{SommeseWampler:2005}. 
Starting from $t_0 =0$ and an approximation $x_0 \approx x(0)$, such methods will next determine some point $t_1 \in (t_0, 1]$ along with an approximation $x_1 \approx x(t_1)$. Iterating this process, one may obtain a sequence of discrete time steps $0 = t_0 < t_1 < \dots < t_k = 1$ and corresponding approximations of the local solution path $x_0, x_1, \dots, x_k\in \CC^n$.

\emph{Certified homotopy path tracking} algorithms rigorously guarantee correctness by enclosing the path $x(t)$ within a compact region where $H(x; t^*) = 0$ has a unique solution for all $t^* \in [0,1]$. This region is constructed by computing discrete time steps $0 = t_0 < t_1 < \cdots < t_k = 1$, along with interval boxes $I_i$ that contain the solution path $x(t)$ uniquely over each subinterval $[t_{i-1}, t_i]$.

From a practical perspective, the main challenge of certified homotopy path tracking is that the associated step-sizes $t_{i} - t_{i-1}$ will typically be much smaller than those used by non-certified methods.
This challenge has motivated a large amount of work on certified tracking, predominantly focused on two classes of methods. One class of methods \cite{beltran2012certified,beltran2013robust,hauenstein2014posteriori} relies on Smale’s $\alpha$-theory~\cite[Chapter 8]{blum1998complexity}, while another \cite{doi:10.1137/0731048,martin2013certified,van2011reliable,xu2018approach,duff2024certified,guillemot2024validated,guillemot2025certified} uses interval arithmetic to validate approximate solutions and construct certified enclosures of the path. 

In this paper, we closely follow the interval-based method of~\cite{guillemot2024validated}, which empirically demonstrated a major improvement over previous certified path-trackers.
This method has three main ingredients:

\begin{enumerate}
    \item \textbf{Prediction: } The core idea of any certified path tracking algorithm is to construct a series of interval boxes $I_1,\dots, I_k$ that begin from one point on (or near) the path $x(t)$ and cover a portion of the path uniquely. To construct such an interval box, we use a numerical approximation $x\approx x(t)$. If the Krawczyk test (\Cref{algo:Krawczyk-test}) passes for a suitable region containing $x$, we proceed to the next step of tracking.

    Given a solution path $x(t)$ and an approximation $x \approx x (t^*)$ for some $t^*\in (0,1)$, a \emph{predictor} is simply some mapping $X(\eta):[t^*,\infty)\to \mathbb{C}^{n}$ with $X(t^*)=x$, constructed to approximate $x(t)$ for $t$ near $t^*.$
    Perhaps the simplest possible choice is the order-$0$ constant predictor; $X(\eta ) = x$ for all $\eta \ge t^*.$
    A more effective class of higher-order predictors, used for certified homotopy tracking in~\cite[Section 6.3]{guillemot2024validated} and~\cite{van2011reliable}, is based on the notion of Taylor model~\cite[Section 9.3]{moore2009introduction}, which can be used to construct a curved interval box containing a higher-order approximation of $x(t)$; for example, the Hermite predictor defined in equation~\eqref{eq:hermite-predictor} below.

\item \textbf{Refinement: }
\Cref{algo:meta_refine} formalizes a refinement procedure, which attempts to improve $x$, a $\tau$-approximate solution of some system $F$, by replacing it with a $\rho$-approximate solution with $\rho \ll \tau .$
This enables larger step-sizes to be chosen in future predictor steps.  

\begin{algorithm}[ht]
\caption{Refine (cf.~\cite[Algorithm 2]{guillemot2024validated})}
\label{algo:meta_refine}
\begin{algorithmic}[1]
\Require  
\begin{itemize}
    \item A square system $F=\{f_1,\dots, f_n\}$ of polynomial or rational functions,
    \item a $\tau$-approximate solution $x\in \mathbb{C}^n$ to $F$ with certification radius $r$ (hence $\tau , r \in (0,1)$),
    \item an $n\times n$ invertible matrix $A$,
    \item a contraction factor $g \in (0, 1-\tau)$, and
    \item a target constant $\rho\in (0,1)$ (a parameter used in~\Cref{algo:Krawczyk-test}; typically $\rho < \tau$).
\end{itemize}
\Ensure A $\rho$-approximate solution $x$, a certification radius $\tilde{r}$, and an invertible matrix $\tilde{A} \in \CC^{n\times n}$.
\State{Set $\tilde{r}=r$}
\State{Set $\tilde{A}=A$} 
\While{$\text{KrawczykTest}(F, x, \tilde{r}, \tilde{A}, \rho)=false$}\label{line:while_loop_in_Refine-1}
\If{$\|A F(x)\|\leq g \rho \tilde{r}$}
\State{Set $\tilde{r}=\frac{1}{2}\tilde{r}$}
\Else
\State{Set $x=x-AF(x)$}
\EndIf
\State{Set $\tilde{A}=JF(x)^{-1}$}
\EndWhile
\While{$2\tilde{r}\leq 1$ and $\text{KrawczykTest}(F, x, 2\tilde{r}, \tilde{A}, \rho)=true$}\label{line:while_loop_in_Refine-2}
\State{Set $\tilde{r}=2\tilde{r}$.}
\EndWhile
\State{Return $x,\tilde{r},\tilde{A}$.}
 \end{algorithmic}
 \end{algorithm}

The first while loop in~\Cref{algo:meta_refine} (starting on line~\ref{line:while_loop_in_Refine-1}) reduces the radius at each iteration, potentially resulting in a narrow region. When the interval becomes overly narrow, the solution path is more likely to drift outside the box, making reliable tracking difficult. The second while loop (line~\ref{line:while_loop_in_Refine-2}) attempts to mitigate this effect, using improved estimates obtained from the first loop to certify a larger radius for the Krawczyk test.

\Cref{algo:meta_refine} is more-or-less equivalent to \cite[Algorithm 2]{guillemot2024validated}, where specific values for the parameters $g,\rho$ and $\tau$ were chosen to balance good practical performance with a theoretical analysis of termination under an adaptive precision model~\cite[Section 4.2]{guillemot2024validated}. 

    \item \textbf{Tracking: } A fully certified path tracking algorithm for a homotopy $H(x;t)$ proceeds from a certified approximate solution $x$ to the start system $H(x;0)$,
    alternating between the predictor and refinement steps described above. This is outlined in~\Cref{algo:certified_curve_tracking}.

\begin{algorithm}[ht]
	\caption{CertifiedTrack}
 \label{algo:certified_curve_tracking}
\begin{algorithmic}[1]
\Require  
\begin{itemize}
    \item A homotopy $H(x;t):\mathbb{C}^{n}\times [0,1]\to \mathbb{C}^n$ with a regular solution path $x(t)$,
    \item constants $\rho\in(0,\frac{1}{2}]$ and $\tau\in(\frac{1}2,1)$,
    \item a $\rho$-approximate solution $x\in\mathbb{C}^n$ to $H(x;0)$,
    \item a predictor $X(\eta)$ (possibly with additional arguments),
    \item a constant $0<h<1$,
    \item a constant $0<r<1$, and
    \item constants $\delta_{inc}>1$ and $0<\delta_{dec}<1$.
\end{itemize}
\Ensure {A $\rho$-approximate solution $x$ of $H(x;1)$.}
\State{Set $t=0$,}
\State{Set $A= \partial_xH(x;t)^{-1}$.}
\While{$t<1$} \label{line:track-outer-while}
\State{Set $(x, r, A) =\text{Refine}(H(x;t),x,r,A,\rho)$.}\label{line:refine} \Comment{\Cref{algo:meta_refine}}
\State{Set $h=\delta_{inc}h$.} 
\While{$\text{KrawczykTest}(H(x;[t,t+h]),X([t,t+h]),r,A,\tau)=false$}\label{line:10}
\State{$h=\delta_{dec}h$.}
\EndWhile
\State{Set $t= \min ( t+h, 1)$.}
\EndWhile
\State{Set $(x, r, A) =\text{Refine}(H(x;1),x,r,A,\rho)$.}\label{line:6}
\State{Return $x$.}
 \end{algorithmic}
 \end{algorithm}
 
Within the outer while loop (\ref{line:track-outer-while}), we refine a previous approximate solution $x$ in order to obtain a $\rho$-approximate solution to the homotopy $H(x;t)$ at time $t$. The inner while loop (\cref{line:10}) determines a suitable time-step $h$ by applying the Krawczyk test (\Cref{algo:Krawczyk-test}), so that $x$ remains an approximate solution to the homotopy at time $t+h.$

\Cref{algo:certified_curve_tracking} is heavily inspired by \cite[Algorithm 4]{guillemot2024validated}, which makes particular choices for the predictor $X$ and the step increase/decrease factors $\delta_{inc}, \delta_{dec}.$
Theoretically, this algorithm will produce an approximate solution to the target system $H(x;1)=0$, assuming that the start solution extends to a smooth solution path and that all arithmetic operations are exact.
In practice, one must of course reckon with inexact fixed-precision floating-point arithmetic.
A careful theoretical analysis of termination for a version of this algorithm has been provided in~\cite[Section 6.3]{guillemot2024validated}, assuming inexact arithmetic with adaptive precision.

On the other hand, even for inexact fixed-precision arithmetic, we can be sure that~\Cref{algo:certified_curve_tracking} certifies a solution path, assuming that it does actually terminate as well as certain properties of the implementation of interval arithmetic (e.g.~conservative outward rounding of endpoints).
In particular, assuming that tracking terminates, the output $x$ will be a certified approximate solution to the target system $H(x;1)=0$.
In our experiments, we always work with fixed precision, chosen sufficiently large so that all tracked paths terminate.
\end{enumerate}

\section{Certified monodromy computation with homotopy graphs}\label{sec:homotopy-graph}

In this section, we elaborate on the process of computing the monodromy action using parameter homotopy, in which the resulting monodromy permutations are obtained in a fully certified manner.

\subsection{Parameter homotopies and edge correspondences}

In general, monodromy heuristics must begin by calling some \emph{fabrication procedure}, capable of producing a (sufficiently generic) problem-solution pair $(z_0, x_0) \in X\subset \mathbb{C}^m\times \mathbb{C}^n$ (or, more generally, one such pair for every component of $X$ dominating the parameter space $\mathbb{C}^m$). 
Although the exact fabrication strategy used is typically problem-dependent, we assume such a strategy is given in addition to $X$ as part of the input.
This assumption, perhaps seemingly restrictive, is quite natural for many problems in science and engineering, where solving the equations~\eqref{eq:Fxz} in $x$ for a given $z$ is the \emph{inverse problem} relative to the more tractable \emph{forward problem} of producing $z$ from $x.$

We further require a system of equations $F = \{f_1(x;z), \dots, f_n(x;z)\}$ which locally defines $X$ near $(z_0, x_0)$. More precisely, the system $F$ may consist of polynomial or rational  functions,
and we require that the Jacobian matrix with respect to the solution variables,
\[
\partial_xF(x_0;z_0) = \left(\frac{\partial f_i}{\partial x_j}\right)_{i,j=1}^n\bigg|_{(z_0, x_0)},
\]
has full rank $n$ for any generic sample $(z_0, x_0) \in X$. This ensures, by the implicit function theorem, the solution variables can be locally expressed as a smooth function $x = g(z)$ near $(z_0, x_0)$.

Given a fabricated problem-solution pair $(z_0, x_0) \in X$, we generate a homotopy graph (defined below) on the parameter space. To create the vertices in the homotopy graph, we generate $\ell$ new generic problem instances $z_i \in \mathbb{C}^m$ for $i=1, \dots, \ell$. To form the edges of our homotopy graph, we consider paths between pairs of problem instances $(z_i, z_j)$ for $i, j = 0, \dots, \ell$. Each path $p(t)$ is defined as the linear segment $p(t) := (1-t)z_i + t z_j$. We assume that for all $t \in [0, 1]$, the fiber $\pi^{-1}(p(t))$ consists of $d$ distinct, isolated solutions. For each such path $p(t)$, we define the corresponding \emph{parameter homotopy} $H(x; t) := \{f_1(x; p(t)), \dots, f_n(x; p(t))\}$. Each edge $e=(z_i,z_j)\in E$ is equipped with its associated parameter homotopy. The collection of vertices $V=\{z_0,\dots,z_\ell\}$ and edges $E$ given by the chosen paths $p(t)$ form a complete graph, and denoted by $\mathcal{G} = (V,E)$. 

More generally, a \emph{homotopy graph} $\mathcal{G}$ may be any connected multigraph as in~\cite{duff2019solving}, allowing the possibility that two vertices are connected by any number of edges. In the case of parallel edges between the same pair of vertices, these should be decorated with different analytic connecting paths $p_e(t)$ (not all straight lines!).
In all cases, it is required that the fiber $\pi^{-1}(p_e(t))$ consists of $d$ distinct, isolated solutions for all $t\in[0,1]$, i.e.~each $p_e(t)$ avoids the discriminant. 

In the constructed homotopy graph, tracking along an edge $e$ from a parameter instance $z_i$ with a known solution $x_i$ to a neighbor $z_j$ yields a new solution $x_j$. In this case, we obtain the new problem-solution pair $(z_j,x_j)$. Also, we record the specific correspondence from $x_i$ to $x_j$ associated with the edge $e.$ We refer to this recorded relationship as an \emph{edge correspondence}.

As an illustrative example, let us consider a $3$-cycle consisting of vertices $z_0, z_1, z_2$. Using the constructed parameter homotopies, we track the path from $z_0$ to $z_1$ to obtain a solution $x_1$, establishing the correspondence along the first edge. Proceeding further, we track from $z_1$ to $z_2$ to find a solution $x_2$, and finally track back to $z_0$.
In this manner, the $3$-cycle defines a \emph{monodromy loop} that induces a permutation of the fiber $\pi^{-1}(z_0)$.
This implies that this procedure may yield a solution in $\pi^{-1}(z_0)$ distinct from the initial $x_0$. In such cases, the sequence of recorded edge correspondences allows us to explicitly reconstruct how the monodromy action was formed along the cycle. Consequently, as more points in the fiber are discovered, the permutations induced by the cycles of the homotopy graph generate the corresponding portion of the monodromy group.

In principle, recovering the full monodromy group does not require establishing all edge correspondences; it suffices to collect a set of loops whose induced permutations generate the group (see \Cref{sec:27-lines}). However, since the order and structure of the target group are often unknown \textit{a priori}, it is desirable to discover the complete fiber. A rigorous stopping criterion for checking fiber completeness is the multihomogeneous trace test \cite{leykin2018trace}. However, this stopping criterion can be inefficient, as it requires solving an auxiliary system with many solutions. In practice, a common heuristic is to terminate the computation once a prescribed number of successive monodromy loops fail to produce any new solutions in the fiber.

We say an edge $e=(u, v)$ in a homotopy graph is \emph{saturated} if every discovered solution in $\pi^{-1} (u)$ has been tracked to a solution in $\pi^{-1} (v)$ and vice-versa.
A homotopy graph $\mathcal{G}$ is said to be saturated if all of its edges are saturated. 
Assuming that all paths have been certifiably tracked using~\Cref{algo:certified_curve_tracking}, this implies that the number of solutions found at each vertex is some constant $k.$ 
Ideally, and often in practice, we find that $k=d$ is the true root count of the parametric system.
We emphasize that although a non-saturated graph may be preferable for the purposes of efficiently finding solutions, saturating a homotopy graph tends to increase probability of recovering the full monodromy group.

\subsection{Generating monodromy permutations from a homotopy graph}

Given a saturated homotopy graph with $k$ solutions at each vertex, we extract a monodromy action as follows: 
\begin{enumerate}
    \item For a chosen base vertex $z_0$, we construct a spanning tree $\mathcal{T}$ rooted at $z_0$.
    \item For each vertex $z$ in $\mathcal{G}$, composing the edge correspondences along the unique path in $\mathcal{T}$ from $z_0$ to $z$ yields a bijection 
\[
\phi_{\mathcal{T},z} : \pi^{-1}(z_0) \to \pi^{-1}(z).
\]
\item Each edge $e = (u,v)$ that does not belong to the spanning tree $\mathcal{T}$ closes a
unique fundamental cycle with $\mathcal{T}$, and each such cycle yields one monodromy permutation. The corresponding monodromy permutation on $\pi^{-1}(z_0)$ is
given by
\begin{equation}\label{eq:tree-permutations}
\sigma_{\mathcal{T}, e} = \phi_{\mathcal{T},v}^{-1} \circ \gamma_e \circ \phi_{\mathcal{T},u},
\end{equation}
where $\gamma_e : \pi^{-1}(u) \to \pi^{-1}(v)$ denotes the edge
correspondence associated with $e$.
\item The collection of permutations $\{\sigma_{\mathcal{T},e}\}$ obtained in this way, as $e$ ranges over all non-tree edges of $\mathcal{G}$, generates a permutation group encoded by the homotopy graph $\mathcal{G}.$
\end{enumerate}

We formalize this process in \Cref{algo:group_generation}.

\begin{algorithm}[ht]
	\caption{MonodromyGroup}
 \label{algo:group_generation}
\begin{algorithmic}[1]
\Require  
\begin{itemize}
    \item A saturated homotopy graph $\mathcal{G} = (V, E)$ with $k$ solutions and a base vertex $z_0 \in V$,
    \item for each edge $e=(u, v) \in E$, a bijective correspondence $\gamma_e: \pi^{-1}(u) \to \pi^{-1}(v)$.
\end{itemize}
\Ensure{Generators $S \subset S_k$ for the action of $\Mon_{\mathcal{G}, z_0} \subset \Mon_\pi $ on $\pi^{-1} (z_0)$}
\State Initialize the set of generators $S = \emptyset$.
\State Construct a spanning tree $\mathcal{T}=(V,E_\mathcal{T})$ of $\mathcal{G}$ rooted at $z_0$.
\State Identify the set of off-tree edges $E_{\text{off}} = E \, \setminus \, E_\mathcal{T}$.
\For{each edge $e = (u, v) \in E_{\text{off}}$}
    \State Find $\phi_{\mathcal{T},u}$ and $\phi_{\mathcal{T},v}$ along the unique $uv$-path in $\mathcal{T}$.
    \State Construct the permutation $\sigma_{\mathcal{T},e} = \phi_{\mathcal{T},v}^{-1} \circ \gamma_e \circ \phi_{\mathcal{T},u}$ on $\pi^{-1}(z_0)$ as in~\eqref{eq:tree-permutations}.
    \State $S= S \cup \{ \sigma_{\mathcal{T},e} \}$.
\EndFor
\State \Return $S$.
 \end{algorithmic}
 \end{algorithm}

The homotopy graph $\mathcal{G}$ encodes a transitive permutation group $\Mon_ {\mathcal{G}, z_0} \hookrightarrow S_k$, generated by all permutations induced by loops obtained from the infinite collection of all closed walks in $\mathcal{G}$ based at $z_0.$
The straightforward graph-theoretic lemma proven below identifies a finite generating set for this group, consisting all permutations $\sigma_{\mathcal{T},e}$ defined by~\eqref{eq:tree-permutations}.

\begin{lemma}
The permutations produced by~\Cref{algo:group_generation} generate the group $\Mon_ {\mathcal{G},z_0}$.
\end{lemma}

\begin{proof}
Consider any closed walk based at $z_0,$ traversing edges $e_1=(u_1, v_1), \dots , e_k=(u_k, v_k)$, where $u_{i+1}=v_i$ for $1\le i <k,$ and $u_1 = v_k = z_0.$
This walk encodes the permutation
\begin{equation}\label{eq:perm-decomp}
\gamma_{e_k} \circ \cdots \circ \gamma_{e_1} 
=
\left( 
\phi_{\mathcal{T},v_k}^{-1}\circ  \gamma_{e_k} \circ \phi_{\mathcal{T},u_k}
\right)
\circ \cdots 
\circ
\left( 
\phi_{\mathcal{T},v_1}^{-1}\circ  \gamma_{e_1}\circ  \phi_{\mathcal{T},u_1}
\right)
.
\end{equation}
For each $e_i$, there are two cases to consider:
\begin{enumerate}
    \item $e_i \notin E_\mathcal{T}$, in which case  $\phi_{\mathcal{T},v_i}^{-1}\circ  \gamma_{e_i} \circ \phi_{\mathcal{T},u_i}= \sigma_{\mathcal{T},e_i}$ as defined in~\Cref{algo:group_generation}, or 
    
    \item $e_i\in E_\mathcal{T},$ in which case $\phi_{\mathcal{T},v_i}^{-1}\circ  \gamma_{e_i} \circ \phi_{\mathcal{T},u_i}=\text{id}.$
\end{enumerate}
We conclude that the permutation~\eqref{eq:perm-decomp} lies in the group generated by all $\sigma_{\mathcal{T},e}$ with $e\notin E_\mathcal{T}.$
\end{proof}

\begin{remark}
We can adapt~\Cref{algo:group_generation} to work with non-saturated graphs by constructing the subgraph consisting of those edges whose correspondences are defined on all $k$ solutions, and then replacing $\mathcal{G}$ with the connected component of this subgraph containing $z_0$.
\end{remark}

\subsection{Certified monodromy computation}

Combining \Cref{algo:certified_curve_tracking} for certified tracking with the combinatorial permutation-extraction of~\Cref{algo:group_generation}, we obtain a complete algorithm for certified monodromy computation. 
\Cref{fig:main-figure} provides an illustration of this algorithm.
We refer to~\Cref{algo:certified_monodromy} for complete pseudocode, and the following statement of correctness.

\begin{thm}\label{thm:correctness}
    Under the hypotheses of Theorem~\ref{thm:Krawczyk-short} applied throughout the execution of Algorithm~\ref{algo:certified_monodromy}, the output set $S$ consists of permutations induced by loops in parameter space.
    The transitive action of $\langle S \rangle \hookrightarrow S_k$ coincides with the action of the group $\Mon_{\mathcal{G},z_0}$ on the $k$ known solutions at $z_0.$ 
    In particular, if $k=d$ is the generic root count, we have the relationships
    $$\langle S \rangle = \Mon_{\mathcal{G},z_0} \subset \Mon_\pi.$$ 
\end{thm}

\begin{remark}\label{remark:saturated-action}
If $\mathcal{G}$ is not saturated, then~\Cref{thm:correctness} may be amended with the weaker conclusion that the action of $\langle S \rangle $ on the $k$ solutions coincides with the action of some subgroup of $\Mon_{\mathcal{G},z_0}$.
\end{remark}

We point out that, in general, $\Mon_{\mathcal{G},z_0}$ may not coincide with the full monodromy group $\Mon_\pi$. Certifying equality requires, for instance, verification that $k=d$, i.e.~the fiber of $z_0$ is fully recovered. In practice, one often relies on heuristic evidence: if the group $\Mon_{\mathcal{G},z_0}$ stabilizes as the graph $\mathcal{G}$ grows, this may provide some support for the claim that the recovered group $\langle S \rangle $ equals $\Mon_\pi$.

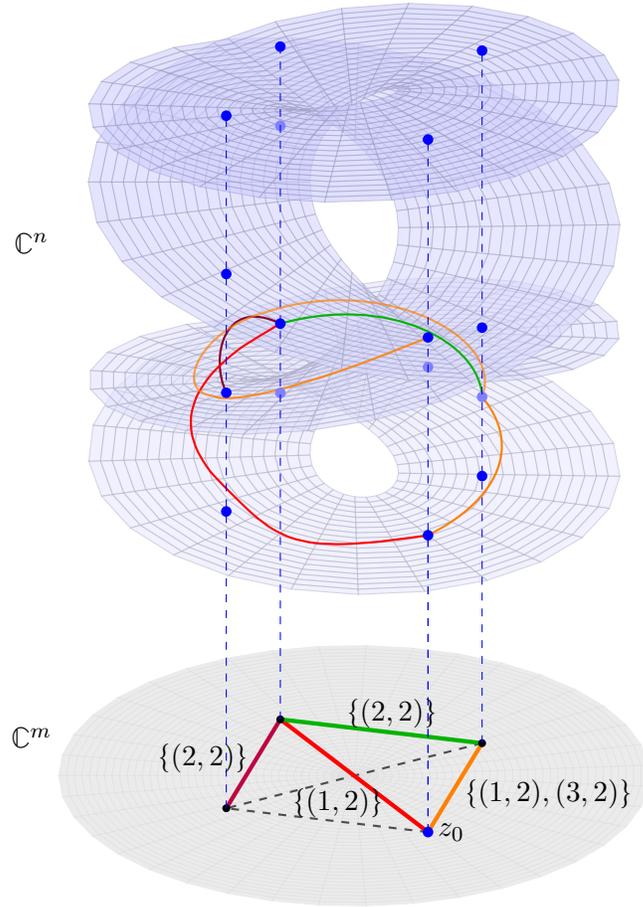
\begin{figure}
    \centering
\input{figure}
\caption{Illustration of certified monodromy computation for a $4$-to-$1$ projection $\pi:X\to\mathbb{C}^m$. The gray circle depicts the parameter space $\mathbb{C}^m$ with a homotopy graph (the base vertex $z_0$ is colored in blue). For each edge (a dashed line), certified path tracking lifts the parameter path to a curve on the incidence variety $X\subset \mathbb{C}^m\times\mathbb{C}^n$. This produces an edge correspondence, which is recorded to obtain monodromy permutations. Edges of the homotopy graph are colored once their correspondences have been computed; the colored curves indicate the corresponding certified lifts.}
    \label{fig:main-figure}
\end{figure}

\begin{algorithm}[ht]
\caption{CertifiedMonodromyComputation}
\label{algo:certified_monodromy}
\begin{algorithmic}[1]
\Require
\begin{itemize}
    \item A system $F(x;z)$ locally defining a solution variety $X \subset \mathbb{C}^m\times\mathbb{C}^n$,
    \item a fabricated generic pair $(z_0,x_0)\in X$.
    \item A connected multigraph $\mathcal{G}$ with  $V(\mathcal{G}) = \{z_0,\dots,z_\ell\}\subset  \CC^m$ and paths $p_e(t)$ for all $e\in E(\mathcal{G})$
\end{itemize}
\Ensure{Generators $S \subset S_k$ for the action of $\Mon_{\mathcal{G}, z_0} \subset \Mon_\pi $ on $\pi^{-1} (z_0)$.}

\While{$\mathcal{G}$ is not saturated}
    \For{each edge $e=(u,v)$ in $\mathcal{G}$}
        \If{$e$ is not saturated}
            \State  \text{CertifiedTrack}  to transport all known solutions at $u$ to $v$ along $e$. \Comment{\Cref{algo:certified_curve_tracking}}
            \State Record each resulting edge correspondence at $e$.
        \EndIf
    \EndFor
\EndWhile
\State Set $S = \text{MonodromyGroup}(\mathcal{G})$. \Comment{\Cref{algo:group_generation}}

\State \Return $S$.
\end{algorithmic}
\end{algorithm}

\begin{remark}\label{rem:compute-monodromy-reducible}
We may  extend~\Cref{algo:certified_monodromy} to the case where $X$ has multiple irreducible components dominating the parameter space if a fabricated pair is provided for each of these components.
\end{remark}

Even in scenarios where~\Cref{algo:certified_monodromy} fails to recover all $d$ solutions, we may still obtain useful information about the full Galois/monodromy group $\Mon_\pi $. 
To illustrate this, we recall the notion of \emph{Galois width} introduced in~\cite{duff2025galois}.
For any finite group $G,$ the invariant $\gw (G)$ provides a measure of the structural complexity of $G$, and may be defined as follows:
\begin{equation}\label{eq:gw}
\gw(G):=\min_{\substack{\text{unrefinable subgroup chains}\\id=H_m\leq \cdots \leq H_0=G}}\left(\max_{0\leq i<m}[H_i:H_{i+1}]\right).
\end{equation}
We now observe that the Galois width is monotone with respect to~\Cref{algo:certified_monodromy}'s output.

\begin{proposition}\label{prop:galois-width}
For any homotopy graph $\mathcal{G},$ we have $\gw (\Mon_{\mathcal{G},z_0} )\le \gw(\Mon_\pi)$.
\end{proposition}
\begin{proof}
As a permutation group, $\Mon_{\mathcal{G},z_0}$ 
acts as the restriction of some subgroup
$H\subset \Mon_\pi$ to its orbit of $k$ solutions connected to $x_0\in\pi^{-1} (z_0)$ through $\mathcal{G}.$
Thus,
$$\gw (\Mon_{\mathcal{G},z_0}) \le \gw (H)\le \gw(\Mon_\pi),$$ where the second inequality follows from~\cite[Theorem 4]{duff2025galois} and, since $\Mon_{\mathcal{G},z_0}$ is a homomorphic image of $H$, the first inequality follows by~\cite[Proposition 3]{duff2025galois}.
\end{proof}

\section{Experiments}\label{sec:experiments}

We have implemented our certified tracking and homotopy graph framework in the \texttt{Julia} programming language. 
Unless otherwise indicated, the homotopy graph in~\Cref{algo:certified_monodromy} is always a complete graph on vertices $v_0, \dots , v_\ell $, whose coordinates are drawn uniformly at random from the complex unit circle, and edges between these vertices are simply straight-line segments.
For interval arithmetic, we use the package \texttt{Nemo.jl} \cite{fieker2017nemo}, which wraps the \texttt{Arb} library~\cite{johansson2017arb}. 
Real and imaginary parts of complex intervals have the form $[x - r, x + r]$, where the radius $r$ has a $30$-bit mantissa, the mantissa of midpoint $x$ has a user-defined precision, and both radius and midpoint have arbitrary-precision exponents.
To ensure that all path-tracking is failure-free, we have fixed $256$ bits of precision throughout all experiments. For many examples, however, much less precision is needed. 
For the parameters of~\Cref{algo:meta_refine,algo:certified_curve_tracking}, we choose the same values used in \cite{guillemot2024validated}, namely $\rho=\frac{1}{8}, \tau=\frac{7}{8},g=\frac{1}{64},h=r=\frac{1}{10},\delta_{inc}=\frac{5}{4}$, and $\delta_{dec}=\frac{1}{2}$, as well as the Hermite predictor $X(\eta)$ defined below. 
An advantage of this predictor is that it exploits information from the previous step. Specifically, let $x$ be the current point (at $\eta=0$) and let $x_{prev}$ be the previous point (at $\eta=-h_{prev}$) obtained during tracking. Let $v$ and $v_{prev}$ denote the corresponding tangent vectors. The \emph{Hermite predictor} $X(\eta)$ is the unique cubic polynomial satisfying
$X(0)=x$, $X(-h_{prev})=x_{prev}$, $X'(0)=v$, and $X'(-h_{prev})=v_{prev}$, 
and is given explicitly (cf.~\cite[\S 6.2]{guillemot2024validated}) by the formula
\begin{equation}\label{eq:hermite-predictor}
X(\eta)=x+v\eta+\left(\frac{2v+v_{prev}}{h_{prev}}-\frac{3(x-x_{prev})}{h_{prev}^2}\right)\eta^2+\left(\frac{v+v_{prev}}{h_{prev}^2}-\frac{2(x-x_{prev})}{h_{prev}^3}\right)\eta^3.
\end{equation}

The first subsection describes a series of experiments conducted on a suite of univariate examples, the results of which are compiled in~\Cref{tab:cert}.
\Cref{tab:results} summarizes experiments described in subsequent subsections.
All code and data are available at the following url:

\begin{center}
\url{https://github.com/klee669/certified_homotopy_tracking}
\end{center}

\begin{table}[t]
    \centering
    \small
\begin{tabular}{l c c l c c}
    \toprule
    Problem & \# variables & \# parameters & Galois group & \# roots & Galois width \\
    \midrule
    Belyi function,~\ref{subsec:mathieu} & $1$ & $1$ & $M_{23}$ & $23$ & $23$ \\
    Symmetric $27$ lines, ~\ref{sec:27-lines} & $8$ & $2$ & $S_2 \times S_2$ & $27$ & $2$ \\
    Nearest-point on a surface,~\ref{subsec:ed-degree} & $4$ & $3$ & $S_4 \wr S_2$ & $8$ & $3$ \\
    3-point absolute pose (P3P),~\ref{subsec:p3p} & $3$ & $15$ & $(S_2 \wr S_4) \cap A_8$ & $8$ & $3$ \\
    $5$-point relative pose,~\ref{subsec:5pp} & $3$ & $20$ & $(S_2 \wr S_{10}) \cap A_{20}$ & $20$ & $10$ \\
    \bottomrule
\end{tabular}

\caption{Summary of experiments in Subsections~\ref{subsec:mathieu}--\ref{subsec:5pp}.}
\label{tab:results} \end{table}

\subsection{Simple univariate families}\label{subsec:univariate}

\begin{table}
    \centering
    \small
    \begin{tabular}{l *{4}{r@{\,}>{\scriptsize}l} @{\hspace{1.5em}} l *{4}{r@{\,}>{\scriptsize}l}}
        \toprule
        \multicolumn{9}{l}{\textbf{Generic coefficients degree: 3}} & \multicolumn{9}{l}{\textbf{Generic coefficients degree: 4}} \\
        \multicolumn{9}{l}{$\text{Mon}_\pi = S_3$, Galois width: 3} & \multicolumn{9}{l}{$\text{Mon}_\pi = S_4$, Galois width: 3} \\
        \cmidrule(r){1-9} \cmidrule(l){10-18}
        \# nodes & \multicolumn{2}{c}{3} & \multicolumn{2}{c}{4} & \multicolumn{2}{c}{5} & \multicolumn{2}{c}{6} & \# nodes & \multicolumn{2}{c}{3} & \multicolumn{2}{c}{4} & \multicolumn{2}{c}{5} & \multicolumn{2}{c}{6} \\
        \midrule
        $\text{Mon}_\pi = \text{Mon}_{\mathcal{G}}$ & 0 & \% & 49 & \% & 78 & \% & 87 & \% & $\text{Mon}_\pi = \text{Mon}_{\mathcal{G}}$ & 0 & \% & 55 & \% & 77 & \% & 94 & \% \\
        $\text{Mon}_{\mathcal{G}} \transitive \pi^{-1}(p)$ & 23 & \% & 68 & \% & 87 & \% & 91 & \% & $\text{Mon}_{\mathcal{G}} \transitive \pi^{-1}(p)$ & 4 & \% & 63 & \% & 78 & \% & 94 & \% \\
        $\text{gw}(\text{Mon}_\pi) = \text{gw}(\text{Mon}_{\mathcal{G}})$ & 23 & \% & 68 & \% & 87 & \% & 91 & \% & $\text{gw}(\text{Mon}_\pi) = \text{gw}(\text{Mon}_{\mathcal{G}})$ & 10 & \% & 74 & \% & 86 & \% & 96 & \% \\
        
        \addlinespace[1.5em]
        
        \multicolumn{9}{l}{\textbf{Generic coefficients degree: 5}} & \multicolumn{9}{l}{\textbf{Generic coefficients degree: 6}} \\
        \multicolumn{9}{l}{$\text{Mon}_\pi = S_5$, Galois width: 5} & \multicolumn{9}{l}{$\text{Mon}_\pi = S_6$, Galois width: 6} \\
        \cmidrule(r){1-9} \cmidrule(l){10-18}
        \# nodes & \multicolumn{2}{c}{3} & \multicolumn{2}{c}{4} & \multicolumn{2}{c}{5} & \multicolumn{2}{c}{6} & \# nodes & \multicolumn{2}{c}{3} & \multicolumn{2}{c}{4} & \multicolumn{2}{c}{5} & \multicolumn{2}{c}{6} \\
        \midrule
        $\text{Mon}_\pi = \text{Mon}_{\mathcal{G}}$ & 0 & \% & 42 & \% & 83 & \% & 91 & \% & $\text{Mon}_\pi = \text{Mon}_{\mathcal{G}}$ & 0 & \% & 29 & \% & 75 & \% & 91 & \% \\
        $\text{Mon}_{\mathcal{G}} \transitive \pi^{-1}(p)$ & 3 & \% & 46 & \% & 85 & \% & 91 & \% & $\text{Mon}_{\mathcal{G}} \transitive \pi^{-1}(p)$ & 4 & \% & 36 & \% & 76 & \% & 93 & \% \\
        $\text{gw}(\text{Mon}_\pi) = \text{gw}(\text{Mon}_{\mathcal{G}})$ & 4 & \% & 46 & \% & 85 & \% & 91 & \% & $\text{gw}(\text{Mon}_\pi) = \text{gw}(\text{Mon}_{\mathcal{G}})$ & 0 & \% & 33 & \% & 76 & \% & 93 & \% \\

        \addlinespace[1.5em]
        
        \multicolumn{9}{l}{\textbf{Generic coefficients degree: 7}} & \multicolumn{9}{l}{\textbf{Sum of even degree terms: Deg 4}} \\
        \multicolumn{9}{l}{$\text{Mon}_\pi = S_7$, Galois width: 7} & \multicolumn{9}{l}{$\text{Mon}_\pi = S_2 \wr S_2$, Galois width: 2} \\
        \cmidrule(r){1-9} \cmidrule(l){10-18}
        \# nodes & \multicolumn{2}{c}{3} & \multicolumn{2}{c}{4} & \multicolumn{2}{c}{5} & \multicolumn{2}{c}{6} & \# nodes & \multicolumn{2}{c}{3} & \multicolumn{2}{c}{4} & \multicolumn{2}{c}{5} & \multicolumn{2}{c}{6} \\
        \midrule
        $\text{Mon}_\pi = \text{Mon}_{\mathcal{G}}$ & 0 & \% & 23 & \% & 74 & \% & 95 & \% & $\text{Mon}_\pi = \text{Mon}_{\mathcal{G}}$ & 0 & \% & 18 & \% & 43 & \% & 64 & \% \\
        $\text{Mon}_{\mathcal{G}} \transitive \pi^{-1}(p)$ & 2 & \% & 28 & \% & 76 & \% & 96 & \% & $\text{Mon}_{\mathcal{G}} \transitive \pi^{-1}(p)$ & 27 & \% & 62 & \% & 89 & \% & 96 & \% \\
        $\text{gw}(\text{Mon}_\pi) = \text{gw}(\text{Mon}_{\mathcal{G}})$ & 2 & \% & 28 & \% & 76 & \% & 96 & \% & $\text{gw}(\text{Mon}_\pi) = \text{gw}(\text{Mon}_{\mathcal{G}})$ & 47 & \% & 75 & \% & 96 & \% & 98 & \% \\

        \addlinespace[1.5em]

        \multicolumn{9}{l}{\textbf{Sum of even degree terms: Deg 6}} & \multicolumn{9}{l}{\textbf{Sum of even degree terms: Deg 8}} \\
        \multicolumn{9}{l}{$\text{Mon}_\pi = S_2 \wr S_3$, Galois width: 3} & \multicolumn{9}{l}{$\text{Mon}_\pi = S_2 \wr S_4$, Galois width: 3} \\
        \cmidrule(r){1-9} \cmidrule(l){10-18}
        \# nodes & \multicolumn{2}{c}{3} & \multicolumn{2}{c}{4} & \multicolumn{2}{c}{5} & \multicolumn{2}{c}{6} & \# nodes & \multicolumn{2}{c}{3} & \multicolumn{2}{c}{4} & \multicolumn{2}{c}{5} & \multicolumn{2}{c}{6} \\
        \midrule
        $\text{Mon}_\pi = \text{Mon}_{\mathcal{G}}$ & 0 & \% & 37 & \% & 73 & \% & 95 & \% & $\text{Mon}_\pi = \text{Mon}_{\mathcal{G}}$ & 0 & \% & 32 & \% & 67 & \% & 95 & \% \\
        $\text{Mon}_{\mathcal{G}} \transitive \pi^{-1}(p)$ & 24 & \% & 56 & \% & 83 & \% & 97 & \% & $\text{Mon}_{\mathcal{G}} \transitive \pi^{-1}(p)$ & 2 & \% & 49 & \% & 70 & \% & 95 & \% \\
        $\text{gw}(\text{Mon}_\pi) = \text{gw}(\text{Mon}_{\mathcal{G}})$ & 27 & \% & 60 & \% & 86 & \% & 99 & \% & $\text{gw}(\text{Mon}_\pi) = \text{gw}(\text{Mon}_{\mathcal{G}})$ & 11 & \% & 83 & \% & 82 & \% & 100 & \% \\
        
        \addlinespace[1.5em]
        
        \multicolumn{9}{l}{\textbf{Even degree (squared coeffs): Deg 4}} & \multicolumn{9}{l}{\textbf{Even degree (squared coeffs): Deg 6}} \\
        \multicolumn{9}{l}{$\text{Mon}_\pi = S_2 \rtimes S_2$, Galois width: 2} & \multicolumn{9}{l}{$\text{Mon}_\pi = S_2^{2} \rtimes S_3$, Galois width: 3} \\
        \cmidrule(r){1-9} \cmidrule(l){10-18}
        \# nodes & \multicolumn{2}{c}{3} & \multicolumn{2}{c}{4} & \multicolumn{2}{c}{5} & \multicolumn{2}{c}{6} & \# nodes & \multicolumn{2}{c}{3} & \multicolumn{2}{c}{4} & \multicolumn{2}{c}{5} & \multicolumn{2}{c}{6} \\
        \midrule
        $\text{Mon}_\pi = \text{Mon}_{\mathcal{G}}$ & 0 & \% & 30 & \% & 68 & \% & 89 & \% & $\text{Mon}_\pi = \text{Mon}_{\mathcal{G}}$ & 0 & \% & 48 & \% & 82 & \% & 96 & \% \\
        $\text{Mon}_{\mathcal{G}} \transitive \pi^{-1}(p)$ & 0 & \% & 30 & \% & 68 & \% & 89 & \% & $\text{Mon}_{\mathcal{G}} \transitive \pi^{-1}(p)$ & 0 & \% & 51 & \% & 82 & \% & 96 & \% \\
        $\text{gw}(\text{Mon}_\pi) = \text{gw}(\text{Mon}_{\mathcal{G}})$ & 56 & \% & 85 & \% & 98 & \% & 100 & \% & $\text{gw}(\text{Mon}_\pi) = \text{gw}(\text{Mon}_{\mathcal{G}})$ & 28 & \% & 80 & \% & 94 & \% & 97 & \% \\

        \addlinespace[1.5em]

        \multicolumn{9}{l}{\textbf{Even degree (squared coeffs): Deg 8}} & \multicolumn{9}{l}{\textbf{Palindrome: Deg 4}} \\
        \multicolumn{9}{l}{$\text{Mon}_\pi = S_2^3 \rtimes S_4$, Galois width: 3} & \multicolumn{9}{l}{$\text{Mon}_\pi = S_2 \wr S_2$, Galois width: 2} \\
        \cmidrule(r){1-9} \cmidrule(l){10-18}
        \# nodes & \multicolumn{2}{c}{3} & \multicolumn{2}{c}{4} & \multicolumn{2}{c}{5} & \multicolumn{2}{c}{6} & \# nodes & \multicolumn{2}{c}{3} & \multicolumn{2}{c}{4} & \multicolumn{2}{c}{5} & \multicolumn{2}{c}{6} \\
        \midrule
        $\text{Mon}_\pi = \text{Mon}_{\mathcal{G}}$ & 0 & \% & 35 & \% & 83 & \% & 97 & \% & $\text{Mon}_\pi = \text{Mon}_{\mathcal{G}}$ & 0 & \% & 47 & \% & 74 & \% & 90 & \% \\
        $\text{Mon}_{\mathcal{G}} \transitive \pi^{-1}(p)$ & 0 & \% & 54 & \% & 87 & \% & 97 & \% & $\text{Mon}_{\mathcal{G}} \transitive \pi^{-1}(p)$ & 10 & \% & 61 & \% & 68 & \% & 93 & \% \\
        $\text{gw}(\text{Mon}_\pi) = \text{gw}(\text{Mon}_{\mathcal{G}})$ & 27 & \% & 81 & \% & 98 & \% & 99 & \% & $\text{gw}(\text{Mon}_\pi) = \text{gw}(\text{Mon}_{\mathcal{G}})$ & 53 & \% & 89 & \% & 97 & \% & 99 & \% \\

        \addlinespace[1.5em]

        \multicolumn{9}{l}{\textbf{Palindrome: Deg 6}} & \multicolumn{9}{l}{\textbf{Palindrome: Deg 8}} \\
        \multicolumn{9}{l}{$\text{Mon}_\pi = S_2 \wr S_3 $, Galois width: 3} & \multicolumn{9}{l}{$\text{Mon}_\pi = S_2 \wr S_4$, Galois width: 3} \\
        \cmidrule(r){1-9} \cmidrule(l){10-18}
        \# nodes & \multicolumn{2}{c}{3} & \multicolumn{2}{c}{4} & \multicolumn{2}{c}{5} & \multicolumn{2}{c}{6} & \# nodes & \multicolumn{2}{c}{3} & \multicolumn{2}{c}{4} & \multicolumn{2}{c}{5} & \multicolumn{2}{c}{6} \\
        \midrule
        $\text{Mon}_\pi = \text{Mon}_{\mathcal{G}}$ & 0 & \% & 34 & \% & 56 & \% & 85 & \% & $\text{Mon}_\pi = \text{Mon}_{\mathcal{G}}$ & 0 & \% & 21 & \% & 66 & \% & 84 & \% \\
        $\text{Mon}_{\mathcal{G}} \transitive \pi^{-1}(p)$ & 7 & \% & 52 & \% & 82 & \% & 91 & \% & $\text{Mon}_{\mathcal{G}} \transitive \pi^{-1}(p)$ & 4 & \% & 44 & \% & 78 & \% & 87 & \% \\
        $\text{gw}(\text{Mon}_\pi) = \text{gw}(\text{Mon}_{\mathcal{G}})$ & 22 & \% & 67 & \% & 77 & \% & 93 & \% & $\text{gw}(\text{Mon}_\pi) = \text{gw}(\text{Mon}_{\mathcal{G}})$ & 14 & \% & 65 & \% & 89 & \% & 97 & \% \\

        \bottomrule
    \end{tabular}
    \caption{Experimental results for univariate families in degrees 3--8.}
    \label{tab:cert}
\end{table}

We first consider several univariate benchmarks of small degrees. In each case, the parameters are the polynomial coefficients, subject to the appropriate symmetry or sparsity constraints, and the monodromy action is transitive. The cases are:

\begin{itemize}
    \item \textbf{Generic coefficients:} $\sum_{i=0}^dc_ix^i$, with full-symmetric monodromy, $\Mon_\pi \cong S_d$,
    \item \textbf{Even degree terms:} $\sum_{i=0}^{d} c_{2i} x^{2i}$, with imprimitive monodromy $\Mon_\pi \cong S_2 \wr S_d \cong S_2^{d} \rtimes S_d$,
    \item \textbf{Even degree terms + squared coefficients:} $\sum_{i=0}^{d} (c_{2i} x^{i})^2$, with $\Mon_\pi \cong S_2^{d-1} \rtimes S_d$,
    \item \textbf{Palindrome:}  $\sum_{i=0}^{d} c_i x^i$ where $c_i = c_{d-i}$, with $\Mon_\pi \cong S_2 \wr S_d $.
\end{itemize}

For each of the cases listed above, and various $2\le d \le 8,$ we constructed complete homotopy graphs on $k=3,4,5,6$ vertices and ran~\Cref{algo:certified_monodromy} over $100$ trials. Our goal was to understand how the graph size and group properties influenced complete or partial recovery.

The results are summarized in~\Cref{tab:cert}. We report three metrics:
\begin{itemize}
    \item $\Mon_\pi = \Mon_{\mathcal{G}}$: The percentage of cases recovering the exact Galois/monodromy group.
    \item $\Mon_{\mathcal{G}} \transitive \pi^{-1} (p)$: The percentage of cases recovering the exact root count.
    \item $\gw(\Mon_\pi ) = \gw(\Mon_{\mathcal{G}})$: The percentage of cases recovering the exact Galois width.
\end{itemize}

Several trends can be observed from the data. 

First, as expected, the success rates for all metrics generally increase with the number of vertices, confirming that denser homotopy graphs tend to recover group properties more reliably.

Second, higher-degree polynomials generally require more nodes to achieve high success rates. 
For the case of generic coefficients, this runs contrary to the probabilistic heuristic given in~\cite{duff2019solving}, suggesting that only a constant number of loops may be needed as $d\to \infty.$
Our experimental results are partly a symptom of working with small values of $d.$
For multivariate systems with many solutions and full-symmetric monodromy, we can often achieve full (albeit non-certified) monodromy recovery after encircling just a few branch points.

Last, but not least, we observe that the success rate for Galois width recovery on the three imprimitive cases often exceeds the success rates for both transitivity and exact group recovery.
In particular, on sparse graphs the method frequently identifies the correct width even without finding all solutions. This highlights the practical value of Galois width, which we can use to measure the intrinsic complexity of a problem well before its solution space has been fully explored.

\subsection{A Belyi function for the Mathieu group}\label{subsec:mathieu}

\begin{figure}[h]
\begin{tikzpicture}[
    vertex/.style={circle, draw=black, fill=black, inner sep=1.5pt, minimum size=4pt},
    branch/.style={circle, fill=red, inner sep=1.5pt, minimum size=4pt},
    thick_edge/.style={thick, draw=black}
]
    \draw[dashed, gray!60] (-3.5, 0) -- (3.5, 0);
    \draw[dashed, gray!60] (0,-3) -- (0,3);
    \node[gray!60] (rez) at (4.25,0)  {$\operatorname{Im} (z)=0$}; 
        \node[gray!60] (rez) at (1, 2)  {$\operatorname{Re} (z)=1/2$}; 
    \node[vertex, label=below:$v_0$] (v0) at (0,0) {};
    \node[vertex, label=above:$v_3$] (v3) at (30:3) {};
    \node[vertex, label=below:$v_4$] (v4) at (-30:3) {};
    \node[vertex, label=above:$v_1$] (v1) at (150:3) {};
    \node[vertex, label=below:$v_2$] (v2) at (210:3) {};
    \draw[thick_edge] (v0) -- (v1) -- (v2) -- (v0);
    \draw[thick_edge] (v0) -- (v3) -- (v4) -- (v0);
    \node[branch, label={[text=red]above:$z=0$}] (b_left) at ({-3/sqrt(3)}, 0) {};
    \node[branch, label={[text=red]above:$z=1$}] (b_right) at ({3/sqrt(3)}, 0) {};
\end{tikzpicture}
\caption{Homotopy graph $\mathcal{G}$ encircling branch points $z\in \{ 0,1\}$ for the Belyi family $f(x;g)=z$ defined by~\eqref{eq:mathieu-polynomials}, with $g$ satisfying~\eqref{eq:g-poly}.}
\label{fig:mathieu-data}
\end{figure}
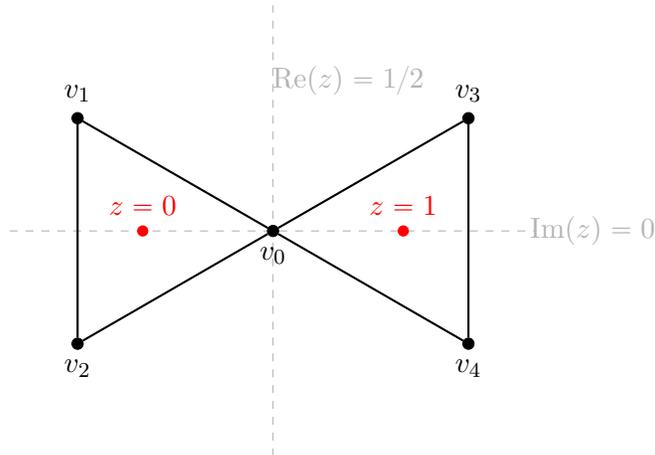

As an additional univariate example, we revisit a degree-$23$ Belyi polynomial constructed by Elkies~\cite{elkies2013complex} whose Galois/monodromy group is the sporadic Mathieu group $M_{23}.$
This example was used to complete a classification begun by M\"{u}ller~\cite{muller1995primitive} of all polynomials $f(x)-z\in \CC [x,z]$ with primitive Galois groups over $\mathbb{C}(z).$
The \emph{inverse Galois problem over $\mathbb{Q}$} provides further motivation for the study of this example. The group $M_{23}$ is the last of the sporadic finite simple groups for which this problem has not yet been solved in the affirmative. 
Also, as of recent work~\cite{van202417t7}, $M_{23}$ is the only permutation group of degree $\le 23$ for which an explicit solution to this problem has not been constructed.

For completeness, we reproduce Elkies' equations below:
\begin{align}
\tau &= (2^{38}  3^{17}/23^3)  (47323 G^3 - 1084897 G^2 + 7751 G - 711002),\nonumber   \\
f_2 (x;G) &= (8 G^3 + 16   G^2 - 20 G + 20)   x^2 - (7 G^3 + 17 G^2 - 7 G + 76) x \nonumber \\
  &\phantom{=}     + (-13 G^3 + 25 G^2 - 107 G + 596),\nonumber \\
f_3 (x;G) &= 8 (31 G^3 + 405 G^2 - 459 G + 333)   x^3 + (941 G^3 + 1303   G^2 -1853 G + 1772) x  \nonumber \\
  &\phantom{=} + (85 G^3 - 385 G^2 + 395 G - 220),\nonumber \\
f_4 (x;G) &= 32   (4 G^3 - 69   G^2 + 74 G - 49)   x^4 + 32 (21 G^3 + 53 G^2 - 68 G + 58) x^3 \nonumber \\
  &\phantom{=} -8 (97 G^3+95 G^2-145 G+148) x^2 + 8   (41 G^3 - 89 G^2-G + 140) x \nonumber \\
  &\phantom{=} +(-123 G^3+391 G^2-93 G+3228),\nonumber \\
f(x;G) &= \tau^{-1} f_2^2(x) f_3 (x) f_4^4 (x) ,\nonumber  \\
\label{eq:mathieu-polynomials}
\end{align}
where we have provided our own normalization of the polynomial $f(x;G)$, and $G$ is the interval
\begin{equation}\label{eq:g-interval}
G = [0.549472304541888876273331 \, \pm \, r ] + i [0.67565033576789561392800491 \,  \pm \, r], 
\quad 
r = 10^{-30},
\end{equation}
which contains a degree-$4$ algebraic number $g$ satisfying
\begin{equation}\label{eq:g-poly}
g^4 + g^3 + 9g^2 - 10g + 8 = 0.
\end{equation}
Our normalization is chosen so that branch points of the \emph{Belyi map} $\CC \ni x \mapsto f(x;g) \in \CC$ are at $z=0$ and $z=1.$ These branch points are visualized in~\Cref{fig:mathieu-data}, along with a bowtie-shaped homotopy graph $\mathcal{G}$ which we used to certify Elkies' result.
As we vary $z$ along the edges of the graph, we obtain a sequence of intervals which contain solution curves $x_g (z)$ satisfying $f(x_g(z))=z$ for all $g\in G.$
For values $g\in G$ not satisfying~\eqref{eq:g-poly}, the branch point $z=1$ splits into clusters of different nearby branch points; for these values, we have merely certified $M_{23}$ as a subgroup of the full Galois/monodromy group.
However, when $g$ does satisfy the equation~\eqref{eq:g-poly}, we have certified the monodromy group of the Belyi polynomial, which only has two branch points by construction.

We close with a quotation about this example from~\cite{elkies2013complex};
\textit{``monodromy calculation would require some careful estimates to guarantee that the precision was
sufficient to obtain the correct permutations."} 
Much to the contrary, the use of interval arithmetic in~\Cref{algo:certified_monodromy} allows us to painlessly certify this monodromy action without the need for any \textit{a priori} estimates.

\subsection{27 lines on a symmetric cubic surface}\label{sec:27-lines}
Classical algebraic geometry tells us there are $27$ lines on a smooth cubic surface $f(x_1,x_2,x_3,x_4)=0$ in $\mathbb{P}^3.$ 
This enumerative problem may be encoded by a degree-$27$ branched cover over the base $\mathbb{P}^{19}$ parametrizing all cubic surfaces.
In recent work \cite{brazelton2024monodromy}, the restriction of this cover to surfaces invariant under the permutations of coordinates in $S_4$ was shown to have the Klein-four Galois/monodromy group $S_2\times S_2$, which acts intransitively on the $27$ lines. 
Here, we provide the details of a computation reported in that work.

The authors of~\cite{brazelton2024monodromy} show that the moduli space of $S_4$-symmetric cubic surfaces is isomorphic to $\mathbb{P}^2$ and that the branch locus of their problem decomposes into four irreducible components: namely, three lines and a cubic curve. By tracking loops encircling these components, the authors generated the full monodromy group, which was certified computationally to be $S_2\times S_2$.

To construct the polynomial system $F$, we represent a line $\ell \subset \mathbb{P}^3$ by two spanning points $u, v \in \mathbb{C}^4$, introducing eight variables
$u=(a_1,b_1,c_1,d_1)$ and $ v=(a_2,b_2,c_2,d_2)$, so that $\ell = \operatorname{span}\{u,v\} = \{su + tv \mid s,t \in \mathbb{C}\}$. The space of $S_4$-symmetric cubic forms admits the basis \cite[Section~2.3]{brazelton2024monodromy}
\[
m_3=\sum_i x_i^3,\quad 
m_{21}=\sum_{i\neq j} x_i^2 x_j,\quad 
m_{111}=\sum_{i<j<k} x_i x_j x_k,
\]
so any symmetric cubic surface can be written uniquely as
\[
f = m_3 + a_{21} m_{21} + a_{111} m_{111},
\]
where we fix the coefficient of $m_3$ to be $1$. The condition $\ell \subset V(f)$ means that the cubic polynomial $f(su+tv)$ vanishes identically in $(s,t)$. Expanding this expression yields a binary cubic whose four coefficients must vanish, giving four cubic equations in the variables of $u$ and $v$. To remove the redundant degrees of freedom associated with the Grassmannian $Gr(2,4)$, we append four generic linear forms. This produces a square system of eight equations in eight variables.

Varying the parameters $(a_{21},a_{111})$ along loops in the parameter space (avoiding the branch locus) transports the $27$ solutions corresponding to the $27$ lines on $V(f)$, and it provides the monodromy action. Altogether, this defines the incidence variety
\begin{equation}\label{eq:27-lines-incidence}
X = \{(a_{21},a_{111},u,v) \mid \ell \subset V(f)\}
    \subset \mathbb{C}^2 \times Gr(2,4),
\end{equation}
and the projection $\pi : X \to \mathbb{C}^2$ is dominant and generically $27$-to-$1$. 
The incidence variety $X$ in this case is reducible, as can be witnessed by the intransitive Galois/monodromy action. 
Specifically, the Galois/monodromy group $S_2\times S_2$ acts faithfully on $3$ orbits consisting of $4$ lines each, has an additional $6$ orbits consisting of $2$ lines each, and fixes the remaining $3$ lines.

\subsection{Nearest point on an algebraic variety}\label{subsec:ed-degree}

The \emph{nearest point problem} for a real algebraic variety $X\subset \mathbb{R}^n$ asks for a point $x\in X$ that minimizes the Euclidean distance from a given data point $u \in \mathbb{R}^n$. The algebraic complexity of this problem is often measured by the \emph{Euclidean distance degree}~\cite{draisma2016euclidean}, which counts the number of complex-valued critical points of the squared distance function on the smooth locus of $X$ for generic data points $u.$ 
Here, we consider the surface in $\mathbb{R}^3$ defined by $f(x_1,x_2,x_3) = x_3^4 - (x_1^2 + x_2^2)^3 = 0$ which previously appeared as~\cite[Example 5]{duff2025galois}.

For the system, the unknowns are the coordinates $x=(x_1,x_2,x_3)$ and a Lagrange multiplier $\lambda$. Given a data point $u \in \mathbb{R}^3$, the critical points are solutions to the augmented system 
$$F(x,\lambda;u)=\{\nabla_x (\|x-u\|^2 + \lambda f(x_1,x_2,x_3)) ,f(x)\}=\left\{\begin{matrix}2(x_1-u_1)-6x_1\lambda(x_1^2+x_2^2)^2\\2(x_2-u_2)-6x_2\lambda(x_1^2+x_2^2)^2\\
2(x_3-u_3)+4\lambda x_3^3\\f(x_1,x_2,x_3)\end{matrix}\right\}=0.$$
These constraints define the Euclidean distance correspondence $X$, which is the incidence variety
\[
X = \{(u,x,\lambda) \mid F(x,\lambda;u)=0 \} \subset \mathbb{C}^3 \times \mathbb{C}^4.
\]
The projection $\pi :X \to \mathbb{C}^3$ is dominant, with $\deg (\pi)$ equal to the Euclidean distance degree, which is known to be $8.$
Using~\Cref{algo:certified_monodromy}, we were able to certify that the Galois/monodromy group for this example contains the imprimitive group $S_4\wr S_2.$
In fact, the reverse containment also holds, since $S_4\wr S_2$ is the largest group consistent with the decomposition of $\pi $ described in~\cite{duff2025galois}.

For this example, the Euclidean distance degree of $8$ is a pessimistic measure of algebraic complexity compared to the much smaller Galois width of $3.$
On the other hand, it is expected that ``most" instances of the nearest-point problem will have full-symmetric monodromy. 
Our proposed framework offers a novel tool for certifying the algebraic difficulty of such problems.

\subsection{P3P: The Perspective-3-Point problem}\label{subsec:p3p}

The \emph{Perspective-$n$-Point} (P$n$P) problem of computer vision asks for the position and orientation of a perspective camera
that maps $n$ known 3D points onto $n$ corresponding 2D points. Among these, the case $n=3$, known as \emph{P3P}, is the minimal case.
We refer the interested reader to the review paper~\cite{haralick1994review}, which describes classical solutions based on solving cubic equations.
We also point out previous work~\cite{duff2022galois} for background on modeling minimal problems in computer vision as branched covers and heuristics for computing their Galois/monodromy groups. 
Here, our experiments demonstrate that such calculations may be rigorously certified, at least in small examples like P3P.

We apply~\Cref{algo:certified_monodromy} to a classical formulation of P3P as three equations in three unknown distances $d_i$ from the $i$-th 3D point, $1 \le i \le 3 ,$ to the unknown camera center. 
Letting $D_{ij}$ denote the squared distance between the $i$-th and $j$-th 3D points, and $\tilde{y}_i \in S^2$ denote unit-length homogeneous coordinates of the $i$-th 2D point, we have 
\begin{equation}\label{eq:grunert}
d_i^2 - 2 u_{ij} d_j + d_j^2 = D_{ij}, \qquad 1 \le i < j \le 3,
\end{equation}
where
\[
u_{ij} := \frac{\langle \tilde y_i , \tilde y_j \rangle}{\|\tilde y_i\| \|\tilde y_j\|}.
\]
Using a stereographic parametrization of $S^2,$ the equations~\eqref{eq:grunert} define a parametric polynomial system $F(d_1,d_2,d_3; z) = 0$ with parameters $z = (q_i, y_i)$ for $i=1,2,3$, and the variety
\[
X = \{(z,d_1,d_2,d_3) \mid F(d;z)=0\} \subset \mathbb{C}^{15} \times \mathbb{C}^3.
\]
The projection $\pi : X \to \mathbb{C}^{15}$ is dominant, and generically $8$-to-$1$ with $\gw (\Mon_\pi) = 3.$
Using~\Cref{algo:certified_monodromy}, we certify that $\Mon_\pi$ contains the imprimitive group $(S_2 \wr S_4) \cap A_8 \cong S_2^3 \rtimes  S_4$.
The imprimitivity is easily understood in this case due to the symmetry on solutions given by $(d_1,d_2,d_3) \mapsto (-d_1, - d_2, -d_3).$
As in the previous example, the reverse inclusion $\Mon_\pi \subset (S_2 \wr S_4) \cap A_8 $ is also known to hold, due to the explicit computation of a square discriminant for this problem in~\cite[\S 3]{duff2022galois}.
Thus, using~\Cref{algo:certified_monodromy}, we have proven that $\Mon_\pi \cong (S_2 \wr S_4) \cap A_8.$

\subsection{5-point relative pose problem}\label{subsec:5pp} 

We end with the \emph{$5$-point relative pose problem}, another well-studied minimal problem in computer vision.
In a typical formulation of this problem, the unknowns are a rotation $R \in SO(3)$ and a translation $t \in \mathbb{C}^3$ (defined up to scale), which together describe the relative camera geometry. Given five pairs of image points $(y_{1,i}, y_{2,i}) \in \mathbb{C}^2 \times \mathbb{C}^2$ for $i=1,\dots, 5$, we lift them to homogeneous coordinates $\tilde y_{1,i}, \tilde y_{2,i} \in \mathbb{C}^3$. The geometric constraints are encoded by the \emph{epipolar constraint}, 
\begin{equation}\label{eq:epipolar-constraint}
\tilde y_{2,i}^\top ([t]_\times R) \, \tilde y_{1,i} = 0, \qquad i=1,\dots,5.
\end{equation}
Here, $[t]_\times$ denotes the skew-symmetric matrix inducing the cross product with $t$. 
Note that the unknowns can be combined into a matrix $E=[t]_\times R$, which is known as the \emph{essential matrix.}
These equations define a polynomial system $F(R, t; z) = 0$ with parameters $z = (y_{1,i}, y_{2,i})$ for $i=1,\dots,5$. The parameter space for the input data is $\mathbb{C}^{20}$, leading to the incidence variety
\[
X = \{(z;R, t) \mid F(R, t; z)=0\} \subset \mathbb{C}^{20} \times (SO(3) \times \mathbb{P}^2).
\]
The projection $\pi : X \to \mathbb{C}^{20}$ is dominant and generically $20$-to-$1$. 
Its monodromy is imprimitive, as there are only $10= \gw (\Mon_\pi)$ distinct essential matrices consistent with generic data, but each essential matrix admits two possible relative poses $(R,t).$ See e.g.~\cite{stewenius2006recent} for more details.

To make certified tracking more tractable, we eliminate the translation variables $t$ by noting that the epipolar constraint~\eqref{eq:epipolar-constraint} is linear in $t$. This can be used to derive three independent cubic constraints on the rotation matrix $R$ alone; specifically, the determinants 
\begin{equation}
\det 
\left(
\begin{array}{c|c|c}
\tilde{y}_{2,1} \times R \tilde{y}_{1,1}
& 
\tilde{y}_{2,2} \times R \tilde{y}_{1,2}
& 
\tilde{y}_{2,i} \times R \tilde{y}_{1,i}
\end{array}\right) =0,
\quad 
i=3,4,5.
\end{equation}
Using Cayley's rational parametrization of $\operatorname{SO} (3),$ 
\begin{equation}
R(x,y,z) = 
\displaystyle\frac{1}{1+x^2+y^2+z^2}  \left(\!\begin{array}{ccc}
     x^{2}-y^{2}-z^{2}+1&2\,x\,y-2\,z&2\,x\,z+2\,y\\
     2\,x\,y+2\,z&-x^{2}+y^{2}-z^{2}+1&2\,y\,z-2\,x\\
     2\,x\,z-2\,y&2\,y\,z+2\,x&-x^{2}-y^{2}+z^{2}+1
     \end{array}\!\right),
\end{equation}
we thus obtain a parametric system in three variables, whose Galois/monodromy group is isomorphic to $\Mon_\pi.$
\Cref{algo:certified_monodromy} certifies that $\Mon_\pi$ contains the imprimitive group $(S_2 \wr S_{10}) \cap A_{20},$ thus strengthening heuristic evidence in~\cite[\S 4]{duff2022galois} for the conjecture that $\Mon_\pi \cong (S_2 \wr S_{10}) \cap A_{20}$.

\bibliography{ref}
\bibliographystyle{abbrv}

\end{document}

%% file: figure.tex
\begin{tikzpicture}
    \begin{axis}[
        view={60}{30},
        axis lines=none,
        z buffer=sort,
        width=15cm, height=15cm,
        xmin=-2.5, xmax=2.5,
        ymin=-2.5, ymax=2.5,
        zmin=-2.7, zmax=3,
        colormap={whiteblue}{color(0cm)=(white!70); color(1cm)=(blue!20)}
    ]

\addplot3[
  surf,
  shader=flat,
  opacity=0.6,
  draw=none,
  domain=0:2*pi,
  domain y=0:2,
  color=gray!25,  
  samples=40,
]
(
  {y*cos(deg(x))},
  {y*sin(deg(x))},
  {-3}
);
    \coordinate (z0) at (axis cs: 1, 0, -3);
    \coordinate (z1) at (axis cs: 0, 1, -3);
    \coordinate (z2) at (axis cs: -1, 0, -3);
    \coordinate (z3) at (axis cs: 0, -1, -3);

    \draw[dashed, thick, black!70] (z0) -- (z3);
    \draw[dashed, thick, black!70] (z1) -- (z3);
    \draw[ultra thick, orange] (z0) -- (z1);
    \draw[ultra thick, red] (z0) -- (z2);
    \draw[ultra thick, green!70!black] (z1) -- (z2);
    \draw[ultra thick, purple] (z2) -- (z3);

    \fill[blue] (z0) circle (2pt) node[right, black] {$z_0$};
    \fill[black] (z1) circle (1.5pt);
    \fill[black] (z2) circle (1.5pt);
    \fill[black] (z3) circle (1.5pt);

    \addplot3 [
        surf,
        opacity=0.6,        
        domain=0.3:1.8,
        domain y=0:8*pi,  
        samples=20,
        samples y=100,      
    ] (
        {x * cos(deg(y))},  
        {x * sin(deg(y))},  
        {2 * (1 - cos(deg(y)/4))}  
    );

    \addplot3[
        thick,
        orange,         
        domain=0:90, 
        samples=30,  
        variable=t      
    ]
    (
        {cos(t)},       
        {sin(t)},       
        {0.5 * t / 90} 
    );
\draw[thick, green!70!black]
  (axis cs:0,1,0.5)
  .. controls (axis cs:-0.2,1.1,0.9) and (axis cs:-0.9,0.8,1.1)
  .. (axis cs:-1,0,1);

\draw[thick, purple!70!black]
  (axis cs:-1,0,1.0)
  .. controls (axis cs:-0.9,-0.5,1.5) and (axis cs:-0.4,-0.9,1.3)
  .. (axis cs:0,-1,1.2);


  \addplot3[
  thick,
  red,
  z buffer=none,         
  smooth
] coordinates {
  (1.000000, 0.000000, 0.000000)
  (0.928320, -0.330273, 0.004883)
  (0.857812, -0.579688, 0.020313)
  (0.780273, -0.761133, 0.047461)
  (0.687500, -0.887500, 0.087500)
  (0.571289, -0.971680, 0.141602)
  (0.423437, -1.026563, 0.210938)
  (0.235742, -1.065039, 0.296680)
  (0.000000, -1.100000, 0.400000)
  (-0.120898, -1.097852, 0.446289)
  (-0.254688, -1.082813, 0.507812)
  (-0.395508, -1.041992, 0.581055)
  (-0.537500, -0.962500, 0.662500)
  (-0.674805, -0.831445, 0.748633)
  (-0.801562, -0.635938, 0.835938)
  (-0.911914, -0.363086, 0.920898)
  (-1.000000, 0.000000, 1.000000)
};

\draw[thick, orange, opacity=0.7]
  (axis cs:0,1,0.5)
  .. controls (axis cs:-0.2,1.3,0.9) and (axis cs:-1.2,0.9,1.1)
  .. (axis cs:-1.2,0,1) 
  .. controls (axis cs:-1.2,-.5,1) and (axis cs:-.4,-1.2,1.05)
  .. (axis cs:0,-1.1,1.2);
\draw[thick, orange, opacity=.85]
  (axis cs:0,-1.1,1.2)
  .. controls (axis cs:.1,-1.1,1.2) and (axis cs:.8,-1.1,1.6)
  .. (axis cs:1,0,2);

    \draw[dashed, blue, thin] (z0) -- (axis cs: 1, 0, 0);
    \fill[blue] (axis cs: 1, 0, 0) circle (2pt) node[right] {};
    \fill[blue] (axis cs: 1, 0, 2) circle (2pt) node[right] {};
    \fill[blue!50] (axis cs: 1, 0, 1.7) circle (2pt) node[right] {};
    
    \draw[dashed, blue, thin] (z0) -- (axis cs: 1, 0, 4);
    \fill[blue] (axis cs: 1, 0, 4) circle (2pt) node[right] {};

    \fill[blue] (axis cs: 0, 1, -0.3) circle (2pt) node[right] {};
    \fill[blue!50] (axis cs: 0, 1, .5) circle (2pt) node[right] {};
    \fill[blue] (axis cs: 0, 1, 1.2) circle (2pt) node[right] {};
    \fill[blue] (axis cs: 0, 1, 4) circle (2pt) node[right] {};
    \draw[dashed, blue, thin] (z1) -- (axis cs: 0, 1, 4);

    \fill[blue!50] (axis cs: -1, 0, .3) circle (2pt) node[right] {};
    \fill[blue] (axis cs: -1, 0, 1) circle (2pt) node[right] {};
    \fill[blue!50] (axis cs: -1, 0, 3) circle (2pt) node[right] {};
    \fill[blue] (axis cs: -1, 0, 3.8) circle (2pt) node[right] {};
    \draw[dashed, blue, thin] (z2) -- (axis cs: -1, 0, 3.8);

    \fill[blue] (axis cs: 0, -1, 0) circle (2pt) node[right] {};
    \fill[blue] (axis cs: 0, -1, 1.2) circle (2pt) node[right] {};
    \fill[blue] (axis cs: 0, -1, 2.4) circle (2pt) node[right] {};
    \fill[blue] (axis cs: 0, -1, 4) circle (2pt) node[right] {};
    \draw[dashed, blue, thin] (z3) -- (axis cs: 0, -1, 4);




    	\node at (axis cs: -.75, -.75, -3) [align=right] {$\{(2,2)\}$};
        	\node at (axis cs: .3, -.3, -3) [align=right] {$\{(1,2)\}$};
        	\node at (axis cs: .9, 1, -3) [align=right] {$\{(1,2),(3,2)\}$};
        	\node at (axis cs: -.7, .7, -3) [align=right] {$\{(2,2)\}$};
        	\node at (axis cs: -1.6, -1.6, -3) [align=right] {$\mathbb{C}^m$};
            \node at (axis cs: -1.6, -1.6, 2) [align=right] {$\mathbb{C}^n$};

    \end{axis}
\end{tikzpicture}